\documentclass[12pt]{amsart}
\usepackage{amsmath, amsfonts, amssymb,color}
\usepackage{stmaryrd,epsfig}
\usepackage{xypic}

\usepackage{hyperref}
\usepackage{pdfsync}
\usepackage{mathrsfs}
\usepackage[margin=1.15in]{geometry}
\def\SL{{\rm SL}}
\def\Fl{{\rm Fl}}

\def\pt{{\rm pt}}

\def\Z{{\mathbb Z}}
\newcommand{\af}{\mathrm{af}}

\newcommand{\bL}{\mathbb{L}}

\newcommand{\cG}{\mathcal{G}}

\newcommand{\cF}{\mathcal{F}}

\newcommand{\cO}{\mathcal{O}}

\newcommand{\Frac}{\mathrm{Frac}}

\newcommand{\ip}[2]{\langle #1\,,\,#2\rangle}
\newcommand{\K}{\mathbb{K}}
\newcommand{\kk}{\kappa}
\newcommand{\la}{\lambda}

\newcommand{\Tor}{\mathscr{T}or}

\newcommand{\Mb}{\overline{\mathcal{M}}}
\newcommand{\euler}[1]{\chi_{{#1}}}
\newcommand{\ev}{\mathrm{ev}}
\newcommand{\gw}[2]{\langle #1 \rangle^{\mbox{}}_{#2}}

\def\Q{{\mathbb Q}}
\def\C{{\mathbb C}}
\def\Gr{{\mathrm{Gr}}}
\def\tGr{{\overline{\Gr}}}
\def\QK{{\rm QK}}
\def\ch{{\rm ch}}

\newtheorem{lemma}{Lemma}
\newtheorem{prop}{Proposition}
\newtheorem{thm}{Theorem}
\newtheorem{remark}{Remark}
\newtheorem{conj}{Conjecture}

\begin{document}
\title{A conjectural Peterson isomorphism in $K$-theory}
\author{Thomas Lam}
\address{Department of Mathematics, University of Michigan, Ann Arbor, MI 48109, USA}
\email{tfylam@umich.edu}

\author{Changzheng Li}
\address{School of Mathematics, Sun Yat-sen University, Guangzhou 510275, P.R. China}
 \email{lichangzh@mail.sysu.edu.cn}

\author{Leonardo~C.~Mihalcea}
\address{460 McBryde Hall, Department of Mathematics, Virginia Tech,
  Blacksburg, VA 24061, USA}
\email{lmihalce@math.vt.edu}

\author{Mark Shimozono}
\address{460 McBryde Hall, Department of Mathematics, Virginia Tech,
  Blacksburg, VA 24061, USA}
\email{mshimo@math.vt.edu}

\maketitle
\begin{abstract}
We state a precise conjectural isomorphism between localizations of the equivariant quantum $K$-theory ring of a flag variety and the equivariant $K$-homology ring of the affine Grassmannian, in particular relating their Schubert bases and structure constants. This generalizes Peterson's isomorphism in (co)homology. We prove a formula for the Pontryagin structure constants in the $K$-homology ring, and we use it to check our conjecture in few situations.
\end{abstract}
\section{The $K$-Peterson conjecture}\label{sec:intro}
The goal of this manuscript is to present a precise conjecture which asserts the coincidence of the Schubert structure constants for the Pontryagin product in $K$-homology of the affine Grassmannian, with those for the quantum $K$-theory of the flag manifold. This is a $K$-theoretic analogue of the celebrated Peterson isomorphism between the homology of the affine Grassmannian and the quantum cohomology of the flag manifold \cite{P, LS, LL}.

Let $G$ be a simple and simply-connected complex Lie group with chosen Borel subgroup $B$ and maximal torus $T$, Weyl group $W$ and affine Weyl group $W_\af = W \ltimes Q^\vee$ where $Q^\vee$ denotes the coroot lattice.  Let $\Lambda$ denote the weight lattice of $G$, so that the representation ring $R(T)$ of $T$ is given by $R(T) \simeq K_T(\pt) \simeq \Z[\Lambda]$.  The torus-equivariant quantum $K$-theory $\QK_T(G/B)$ of the flag variety $G/B$ has as basis over $\Z[[q]] \otimes R(T)$ the Schubert classes $\cO^w$, for $w \in W$, of structure sheaves of Schubert varieties in $G/B$.  The torus-equivariant $K$-homology $K_0^T(\Gr)$ of the affine Grassmannian $\Gr = \Gr_G$ of $G$ has as basis over $\Z[\Lambda]$ the Schubert classes $\cO_x$ of structure sheaves of Schubert varieties in $\Gr$, where $x$ varies over the set $W_\af^- \subset W_\af$ of affine Grassmannian elements.

\begin{conj}\label{conj:constants} Let $ut_\lambda, vt_\mu, wt_\nu \in W_\af^-$, and let $\eta \in Q^\vee$. Assume that $\nu = \lambda + \mu$.
Then \[ c_{ut_\lambda, vt_\mu}^{wt_{\nu + \eta}} = N_{u,v}^{w, \eta} \/ \qquad \text{in}\; K^*_T(\pt) \] where $c$'s are the structure constants in $K_0^T(\Gr)$ with respect to $\cO_x$ and $N$'s are the structure constants in $\QK_T(G/B)$ with respect to $\cO^w$.
\end{conj}

Conjecture \ref{conj:constants} implies that the multiplication in the ring $\QK_T(G/B)$ is finite, and thus it is possible to define it over $\Z[q]$ instead of $\Z[[q]]$. On the affine side, it implies that we have the formula $\cO_{wt_\lambda} \cdot \cO_{t_\nu} = \cO_{wt_{\lambda+\nu}}$ in the $K$-homology ring  $K_0^T(\Gr)$ endowed with the Pontryagin product. Conjecture \ref{conj:constants} can then be alternatively formulated as follows.


\begin{conj}\label{conj:isomorphism}
The $R(T)$-module homomorphism
\begin{align*}
\Psi: K_0^T(\Gr)[\cO_{t_\lambda}^{-1}] & \longrightarrow { \QK_T(G/B)}[q_i^{-1}] \\
\cO_{wt_\lambda}\cO_{t_\nu}^{-1} &\longmapsto q_{\lambda-\nu}\cO^w
\end{align*}
is an isomorphism of $R(T)$-algebras.
\end{conj}
The remainder of this article makes our conventions precise. We will   give the geometric meaning of \cite[Theorems 5.3 and 5.4]{LSS} in Theorem \ref{T:KPeterson}, and provide a precise combinatorial formula of the aforementioned structure constants $c$'s in Theorem \ref{T: coeffc}. This leads to  computational evidence for these conjectures.

\begin{remark}
Ikeda, Iwao, and Maeno \cite{IIM} have recently shown that the  $K$-homology ring $K_0(\Gr_{\SL_n})$ is isomorphic to Kirillov-Maeno's conjectural presentation of the quantum $K$-theory $QK(\Fl_n)$ of complete flag manifold $Fl_n$ after localization.  Their approach is via the relativistic Toda lattice, and the behavior of Schubert classes under their isomorphism is also studied.
\end{remark}
\begin{remark}
 Braverman and  Finkelberg  \cite{BF} showed that the coefficients of Givental's $K$-theoretic $J$-function \cite{givental:onwdvv} for a flag variety are the equivariant characters of the polynomial functions on a Zastava space, which consists of based quasimaps to the flag variety. Moreover, in each homogeneous degree, the functions on a Zastava space are isomorphic to the functions on a transverse slice of a $G$-stable stratum inside another $G$-stable stratum in the affine Grassmannian. Together with the $K$-theoretic reconstruction theorems  \cite{lee.pandharipande,iritani.milanov.tonita},
  this provides a conceptual connection between quantum $K$-theory of flag varieties and $K$-homology of affine Grassmannians.
\end{remark}

\begin{remark}
   In \cite[Corollary 5.10]{he.lam}, it is shown that the  $K$-homology Schubert structure constants determine the 3-point
$K$-theoretic Gromov-Witten invariants of a cominuscule flag variety
$G/P$. However, a direct formula relating the two sets of invariants is
not given.
\end{remark}

\subsection*{Acknowledgements} The authors thank Takeshi Ikeda for explanations of the work \cite{IIM}.  T.L. acknowledges support from the NSF under agreement No. DMS-1464693. C. L. acknowledges supports from the Recruitment Program of Global Youth Experts in China and   the NSFC Grant 11521101.  L.M. acknowledges support from
NSA Young Investigator Awards H98230-13-1-0208 and H98320-16-1-0013, and a Simons Collaboration Grant. M.S. acknowledges support from the NSF grant DMS-1600653.
\section{Quantum $K$-theory of flag varieties}
Let $G$ be a complex, simple, simply-connected Lie group and $B, B^- \subset G$ is a pair of opposite Borel subgroups containing the fixed torus $T := B \cap B^-$. For each element $w \in W$ in the (finite) Weyl group there are the Schubert cells $X(w)^\circ:= BwB/B$, $Y(w):= B^- wB/B$ and the Schubert varieties $X(w):= \overline{B w B/ B}$ and $Y(w) := \overline{B^- w B/ B}$ in the flag manifold $X:= G/B$. Then $\dim_{\C} X(w) = \mathrm{codim}_{\C} Y(w) = \ell(w)$ (the length of $w$). The {\em boundary} of the Schubert varieties is defined by $\partial X(w) = X(w) \setminus X(w)^\circ$ and $\partial Y(w) := Y(w) \setminus Y(w)^\circ$. The boundary is generally a reduced, Cohen-Macaulay, codimension $1$ subscheme of the corresponding Schubert variety.

We briefly recall the relevant definitions regarding the equivariant $K$-theory ring, following e.g. \cite{chriss.ginzburg}. For any (complex) projective variety $Z$ with an algebraic action of a torus $T$, one can define the equivariant $K$-theory ring $K^T(Z)$. This is the ring generated by symbols $[E]_T$ of $T$-equivariant vector bundles $E \to Z$ subject to the relations $[F]_T +[H]_T = [E]_T$ whenever there is a short exact sequence of $T$-equivariant vector bundles $0 \to F \to E \to H \to 0$ on $Z$. The two ring operations on $K^T(Z)$ are defined by \[ [E]_T + [F]_T := [E \oplus F]_T; \quad [E]_T \cdot [F]_T := [E \otimes F]_T \/. \] There is a pairing $\langle \cdot , \cdot \rangle : K^T(Z) \otimes K^T(Z) \to K^T(pt) = R(T)$, where $R(T)$ is the representation ring of $T$, given by \[ \langle [E]_T, [F]_T \rangle = \chi_T ( Z; E \otimes F )=\sum_{i=0}^{\dim Z}(-1)^i \ch_T\big(H^i( Z; E \otimes F )\big)\/; \] here $\chi_T$ denotes the equivariant Euler characteristic and $\ch_T \in R(T)$ denotes the character of a $T$-module. If in addition $Z$ is smooth then any $T$-equivariant coherent sheaf $\mathcal{F}$ on $Z$ has a {\em finite} resolution by equivariant vector bundles, and thus there is a well defined class $[\mathcal{F}]_T \in K^T(Z)$. This identifies the Grothendieck {\em group} $K_T(Z)$ of equivariant coherent sheaves with $K^T(Z)$. For any $T$-equivariant map of projective varieties $f: Z_1 \to Z_2$, there is a well defined push-forward $f_*: K_T(Z_1) \to K_T(Z_2)$ given by $f_*[\mathcal{F}]_T = \sum_{i \ge 0} (-1)^i [R^i f_* \mathcal{F}]_T$; in this language the pairing above is given by $\langle [E]_T , [F]_T\rangle = \pi_* ([E]_T \cdot [F]_T)$ where $\pi: Z \to pt$ is the structure map.

The maximal torus $T$ acts on $X= G/B$ by left multiplication and the Schubert varieties $X(w), Y(w)$ are $T$-stable. Then the structure sheaves of the Schubert varieties determine the Grothendieck classes $\cO_w := [\cO_{X(w)}]_T$ and $\cO^w := [\cO_{Y(w)}]_T$ in the $T$-equivariant $K$-theory ring $K_T(X)$.
We will also need the {\em ideal sheaf} classes $\xi_w:= [\cO_{X(w)}(- \partial X(w) )]_T$ and $\xi^w:= [\cO_{Y(w)} (- \partial Y(w))]_T$ determined by the boundaries of the corresponding Schubert varieties. The ideal sheaf classes are duals of the Schubert classes: \[ \langle \cO_u, \xi^v \rangle = \langle \cO^u, \xi_v \rangle = \delta_{u,v} \/, \] where $\delta$ is the Kronecker delta symbol. We refer to  \cite[\S 3.3]{brion:flagv} or \cite{anderson.griffeth.miller} for the proofs of this.

Motivated by the relation between quantum cohomology and Toda lattice discovered by Givental and Kim \cite{givental.kim, kim:toda}, Givental and Lee \cite{givental:onwdvv,lee:qk} defined a ring which deforms both the (equivariant) $K$-theory and the quantum cohomology rings for the flag manifold $X$. This is the equivariant {\em quantum $K$-theory} ring $\QK_T(X)$. Additively, $\QK_T(X)$ is the free module over the power series ring $K_T(pt)[[q]]= R(T)[[q_1, \dots , q_r]]$ which has a $R(T)[[q]]$-basis given by Schubert classes $\cO^w$ (or $\cO_w$) as $w$ varies in $W$. Here $r$ denotes the rank of $H_2(X)$, which for $X=G/B$ is the same as the number of simple reflections $s_i \in W$. The multiplication is determined by Laurent polynomials $N_{u,v}^{w,d} \in R(T)$ such that
\begin{equation}\label{E:qkmult}
\cO^u \star \cO^v = \sum_{w \in W} N_{u,v}^{w, {d}} q^{d} \cO^w \end{equation} where the sum is over effective degrees $d= \sum_{i=1}^r m_i [X(s_i)] \in H_2(X)_{\ge 0}$ (i.e. each $m_i \ge 0$), and $q^d = \prod_{i=1}^r q_i^{m_i}$. The precise definition of $N_{u,v}^{w,d}$ requires taking Euler characteristic of certain $K$-theory classes on the Kontsevich moduli space of stable maps $\Mb_{0,3}(X,d)$ and over some of its boundary components $\Mb_{\bf d}(X)$ : \[ N_{u,v}^{w,{d}} = \sum (-1)^j \euler{\Mb_{{\bf d}}(X)} (\ev_1^*(\cO^u) \cdot \ev_2^*(\cO^v) \cdot \ev_3^*(\xi_w)) \/. \] (This is unlike the case of quantum cohomology, where boundary components do not contribute to the structure constants.) Because boundary strata are fiber products of moduli spaces with $2$ or $3$ marked points, a standard calculation (see e.g. \cite[\S 5]{buch.m:qk}) shows that:
\begin{equation}\label{E:exp}
N_{u,v}^{w,d}=\sum \gw{\cO^u, \cO^v, \xi_{\sigma_0}}{d_0} \cdot
    \gw{\cO^{\sigma_0},
         \xi_{\sigma_1}}{d_1} \cdot ... \cdot \gw{\cO^{d_{j-2}},
        \xi_{\sigma_{j-1}}}{d_{j-1}} \cdot
   \gw{\cO^{\sigma_{j-1}}, \xi_w}{d_j} \/,\end{equation}
where the sum is over $\sigma_0, ..., \sigma_j \in W$ and multidegrees ${\bf d}= (d_0, \ldots, d_j)$ such that $d_i \in H_2(X)$ are effective and $d_i \neq 0$ for  $i >0$. The notation $\gw{\cO^u, \cO^v,\xi_\kappa}{d}$ stands for the (equivariant) $3$-point K-theoretic Gromov-Witten (KGW) invariant
\begin{equation}\label{E:defKGW} \gw{\cO^u, \cO^v,\xi_\sigma}{d} =  \euler{\Mb_{0,3}(X,{d})} (ev_1^*(\cO^u) \cdot \ev_2^*(\cO^v)\cdot \ev_3^*(\xi_\sigma)), \end{equation} where $\ev_i: \Mb_{0,3}(X,d) \to X$ are the evaluation maps. If one integrates over $\Mb_{0,2}(X,d)$, it gives the $2$-point invariants $\gw{\cO^u,\xi_v}{d}$. Geometrical properties of the evaluation maps studied in \cite[\S 3]{bcmp:qkfinite} imply that the $2$-point KGW invariants $\gw{\cO^u,\xi_v}{d}$ are always $0$ or $1$. Formulas for these invariants, using combinatorially explicit recursions to calculate {\em curve neighborhoods} of Schubert varieties, can be found in \cite[Rmk.~7.5]{buch.m:nbhds}.

If one declares $\deg q_i = c_1(T_X) \cap [X(s_i)] = 2$ then the quantum $K$-theory algebra $QK_T(X)$ has a filtration by degrees, and its associated graded algebra is naturally isomorphic to the quantum cohomology algebra. 
Because KGW invariants are non-zero for infinitely many degrees (e.g. $\gw{\cO^{id} , \cO^{id}, \cO^{id}}{d} $ is the trivial $1$-dimensional $T$-representation for any degree $d$), it is unclear whether the expansion of the product $\cO^u \star \cO^v \in \QK_T(X)$ has finitely many terms. This was conjectured to be true for any flag manifold $G/P$ by Buch, Chaput, Mihalcea and Perrin. The conjecture is true in the case of cominuscule Grassmannians \cite{bcmp:qkfinite} and for partial flag manifolds $G/P$ with $P$ a maximal parabolic group \cite{bcmp:ratcon}.

While there is no algorithm to calculate the structure constants $N_{u,v}^{w,d}$ for $X=G/B$ or for arbitrary flag manifolds $G/P$, there are several particular instances where algorithms are available. In the case of a cominuscule Grassmannian, a ``quantum = classical" statement, calculating KGW invariants in terms of certain K-theoretic intersection numbers on two-step flag manifolds was obtained in \cite{buch.m:qk}; in Lie types different from type A, this uses rationality results from \cite{chaput.perrin:rationality}. As a result, a {\em Chevalley formula}, which calculates the multiplication by a divisor class, was obtained in \cite{buch.m:qk} for the type A Grassmannians, and it was recently extended in \cite{bcmp:qkchevalley} to all cominuscule Grassmannians. In the equivariant context, this formula determines an algorithm to calculate any product of Schubert classes, generalizing the result from quantum cohomology \cite{mihalcea:eqchev}. Formulas to calculate $N_{u,v}^{w,d}$ for $X=G/B$ and $d$ a ``line class", i.e. $d = [X(s_i)]$, were obtained in \cite{li.mihalcea:lines} by making use of the geometry of lines on flag manifolds. In this case they also proved that $N_{u,v}^{w,d}$ satisfy the same positivity property as the one proved by Anderson, Griffeth and Miller \cite{anderson.griffeth.miller} for the structure constants in the (ordinary) equivariant $K$-theory of $X$. There are also algorithms based on reconstruction formula \cite{lee.pandharipande,iritani.milanov.tonita} which in principle can be used to calculate KGW invariants. In practice, however, these lead to quantities which quickly become unfeasible for explicit calculations.

\section{$K$-homology of the affine Grassmannian}
\subsection{$K$-groups for thick and thin affine Grassmannians}
The foundational reference for the thick affine Grassmannian is \cite{Kas} and for the thin affine Grassmannian we use \cite{KumBook} and \cite{kumar:positivity}.

We use notation from Section \ref{sec:intro}.~The (thin) affine Grassmannian $\Gr$ is an ind-finite scheme: it is the union of finite-dimensional projective Schubert varieties $X_w$, for $w \in W_\af^-$ (in analogy with the Schubert varieties $X(w)$ for $G/B$).  The dimension of the complex projective variety $X_w$ is equal to the length $\ell(w)$.  Let $K_0^T(\Gr)$ be the Grothendieck group of the category of $T$-equivariant finitely-supported (that is, supported on some $X_w$) coherent sheaves on $\Gr$.  We have
$$
K_0^T(\Gr) \simeq \bigoplus_{w \in W_\af^-} R(T) \cdot \cO_w
$$
where $\cO_w = [\cO_{X_w}]_T$ denotes the class of the structure sheaf of $X_w$ (cf. \cite[Section 3]{kumar:positivity}).  We call the $R(T)$-module $K_0^T(\Gr)$ the {\it (equivariant) $K$-homology} of $\Gr$. We notice that $\xi_w^{\Gr}:= [\cO_{X_w}(-\partial X_w)]_T$, $w\in W_{\rm af}^-$, form another $R(T)$-basis of $K_0^T(\Gr)$, which we simply denote as $\xi_w$ whenever it is clear from the context.

The thick affine Grassmannian $\tGr$ is an infinite-dimensional non quasicompact scheme: it is a union of finite-codimensional Schubert varieties $X^w$, for $w \in W_\af^-$, of codimension $\ell(w)$.  Let $K^0_T(\tGr)$ be the Grothendieck group of the category of $T$-equivariant coherent sheaves on $\tGr$, defined for example in \cite[Section 3.2]{LSS}.  We have
$$
K_T^0(\tGr) \simeq \prod_{w \in W_\af^-} R(T) \cdot \cO^w
$$
where $\cO^w = [\cO_{X^w}]_T$ denotes the class of the coherent structure sheaf of $X^w$.

Let  $\Fl$ denote the (ind-)affine flag manifold, and $\overline{\Fl}$ denote the thick version. As for the affine Grassmannian, one defines Schubert varieties $X_w \subset \Fl$ and $X^w \subset \overline{\Fl}$ such that $\dim X(w) = \mathrm{codim}~X^w = \ell(w)$. In this case $w$ varies in the affine Weyl group $W_{\rm af}$. Let $\cO_w = [\cO_{X_w}] \in K_0^T(\Fl)$ and $\cO^w := [\cO_{X^w}] \in K^0_T(\overline{Fl})$; we refer to \cite{KS} or \cite{kumar:positivity} for the (rather delicate) details. There are $T$-equivariant projection maps $\pi: \Fl \to \Gr$ and (abusing notation) $\pi:\overline{\Fl} \to \overline{\Gr}$ which are locally trivial $G/B$-bundles. In particular they are flat, and
\begin{equation}\label{E:1} \pi^*\cO_{X^w_{\overline{\Gr}}}=\cO_{X_{\overline{\Fl}}^w} \end{equation}
for any $w\in W_{\rm af}^-$. {  Further, similar arguments to those in the finite case show that for any $w \in W_{\rm af}$,
\begin{equation}\label{E:2} \pi_*\cO_{X_w^{\Fl}}= \cO_{X^{\Gr}_{\pi(w)}} \/, \end{equation} where $\pi(w)$ denotes the image of $w$ in $W_{\rm af}^-$ under the projection map. (See e.g. \cite[Thm.~3.3.4]{brion.kumar:frobenius} for a proof based on Frobenius splitting; or \cite[Prop.~3.2]{bcmp:qkfinite} for an argument based on a theorem of Koll{\'a}r.)}
%
%
%
There is a pairing $\langle \cdot , \cdot \rangle_{\Fl} : K^0_T(\overline{\Fl}) \otimes K_0^T(\Fl) \to R(T)$ defined  in \cite[\S 3]{kumar:positivity} by
\begin{align}\label{eq:pairing1}
 \begin{split}
\langle \cdot , \cdot \rangle_{\Fl}:& K^0_T(\overline{\Fl}) \otimes K_0^T(\Fl) \longrightarrow R(T)\\
& \ip{[\cF]}{[\cG]}_{\Fl} : = \sum_i (-1)^i \chi_T((\Fl)_n, \Tor_i^{\cO_{\overline{\Fl}}}(\cF,\cG)),
 \end{split}
\end{align}
for any classes $[\cF], [\cG]$ such $\cF$ is a $T$-equivariant sheaf on $\overline{\Fl}$ and $\cG$ is a $T-$-equivariant sheaf supported on a finite dimensional stratum $(\Fl)_n$ of the ind-variety $\Fl$. By \cite[Lemma 3.4]{kumar:positivity} this pairing is well defined. In fact, the definition of this pairing extends in an obvious way to any partial flag variety, in particular to the affine Grassmannian $\Gr$. It was proved in \cite[Prop. 3.9]{baldwin.kumar:posII} that the pairing satisfies the property $\langle \cO^u_{\overline{\Fl}}, \xi_v^{\Fl} \rangle = \delta_{u,v}$.
%
We will need the following additional properties of this pairing.

\begin{lemma}\label{lemma:rich} Consider the pairing $\langle \cdot , \cdot \rangle_{\mathcal{X}}$ and take any $u, v \in W_{\af}$ in the case when $\mathcal{X} = \Fl$ and $u,v \in W_{\af}^-$ for $\mathcal{X} = \Gr$. Then  \[  \langle \cO^u, \cO_v \rangle_{\mathcal{X}} = \begin{cases} 1 & \textrm{ if } u \le v \/; \\ 0 & \textrm{ otherwise } \/. \end{cases} \]
\end{lemma}

\begin{proof} Consider first $\mathcal{X} = \Fl$. By definition  we have
\[ \langle \cO^u, \cO_v \rangle_{\mathcal{X}} = \sum (-1)^i \chi_T(X_v, \mathcal{T}or_i^{\cO_{\overline{\mathcal{X}}}} (\cO_{X^u}, \cO_{X_v})) \/. \] According to \cite[Lemma 5.5]{kumar:positivity} all $Tor$ sheaves are $0$ for $i>0$, and by definition $\mathcal{T}or_0^{\cO_{\overline{X}}}(\cO_{X^u}, \cO_{X_v}) = \cO_{X_u^v}$ where $X^u_v:= X^u \cap X_v$ is the Richardson variety (cf. e.g. \S 2 of {\em loc.cit}). According to \cite[Cor. 3.3]{kumar.schwede}, the higher cohomology groups $H^i(X_v^u, \mathcal{O}_{X_v^u}) = 0$ for $i>0$ and since $X_v^u$ is irreducible $H^0(X_v^u, \mathcal{O}_{X_v^u}) = \C$. It follows that the sheaf Euler characteristic $\chi_T(X_v^u, \cO_{X_v^u}) =1$.

{  We now turn to the situation when $\mathcal{X} = \Gr$. Let $u, v \in W_{\rm af}^-$. The same argument as before reduces the statement to the calculation of $\chi_T(X_v^u, \cO_{X_v^u})$ where $X_v^u \subset \Gr$ is the Richardson variety. By definition of Schubert varieties, $\pi^{-1}(X_v) = X_{v w_0} \subset \Fl$ where $w_0 \in W$ is the longest element in the finite Weyl group, and $\pi^{-1}(X^u) = X^u$. It follows that the preimage of the grassmannian Richardson variety is $\pi^{-1}(X_v^u) = X_{v w_0}^u$. Then a standard argument based on the Leray spectral sequence (taking into account that the fiber of $\pi: \pi^{-1}(X_v^u) \to X_v^u$ is the finite flag manifold $G/B$, and that $H^i(G/B, \cO_{G/B}) = 0$ for $i>0$) gives that $H^i(X_{vw_0}^u, \cO_{X_{vw_0}^u}) = H^i(X_v^u, \cO_{X_v^u})$ for all $i$, thus the required Euler characteristic equals $1$, as needed.}
\end{proof}

\begin{lemma}\label{lemma:xiO}   For any $u, v \in W_{\af}^-$, we have
   $${\rm (i)}\quad   \langle \cO^u, \xi_v  \rangle_{\Gr} = \delta_{u,v}; \qquad {\rm (ii)}\quad \cO_v=\sum\nolimits_{w\leq v; w\in W_{\rm af}^-} \xi_w.$$
\end{lemma}
\begin{proof}
  (i) The statement follows from the same arguments as for $\Fl$ in   \cite[Prop. 3.9]{baldwin.kumar:posII}.
  To prove (ii), we write $ \cO_v = \sum_w a_{w,v} \xi_w \/ $, which is a finite sum because
 the class $\cO_v$ is supported on a finite-dimensional variety. By statement (i) and Lemma \ref{lemma:rich},
 \begin{equation*}  a_{w,v} = \langle \cO^w, \cO_v \rangle_{\Gr} =  \begin{cases} 1 & \textrm{ if } w \le v, \\ 0 & \textrm{ otherwise. }\end{cases} \end{equation*}
\end{proof}
\subsection{$K$-Peterson algebra}\label{s:KPet}
The $K$-groups $K_0^T(\Gr)$ and $K^0_T(\tGr)$ acquire dual Hopf algebra structures from the homotopy equivalence $\Gr \simeq \Omega K$, where $K \subset G$ is a maximal compact subgroup and $\Omega K$ is the group of based loops into $K$.~An algebraic model for these Hopf algebras is constructed in \cite{LSS}.  Only the product structure of $K_0^T(\Gr)$, arising from the Pontryagin product $\Omega K \times \Omega K \to \Omega K$ will be of concern to us.

We consider a variation of Kostant and Kumar's $K$-nilHecke ring, the ``small torus" affine $K$-nilHecke ring of \cite{LSS}, which was inspired by the homological analogue \cite{P}.

The affine Weyl group $W_\af$ acts on the weight lattice $\Lambda$
of $T$ by the level-zero action (that is, we take the null root $\delta=0$)
$$
w t_\mu \cdot \la = w\cdot \la \qquad \mbox{for $w\in W$, $\mu\in Q^\vee$ and $\la\in \Lambda$.}$$
Let $I_\af$ denote the vertex set of the affine Dynkin diagram.  Abusing notation, we denote by $\{\alpha_i \mid i \in I_\af\}$ the images of the simple affine roots in $\Lambda$.  In particular, $\alpha_0 = -\theta \in \Lambda$ where $\theta$ is the highest root of $G$.  Let $Q(T) = \Frac(R(T))$ and equip $\K_Q = Q(T) \otimes_{R(T)} \Q[W_\af]$ with product $(p\otimes v)(q\otimes w) = p(v\cdot q) \otimes vw$ for $p,q\in Q(T)$ and $v,w\in W_\af$.  Define
\begin{align}\label{E:Tdef}
  T_i = (1-e^{\alpha_i})^{-1} (s_i - 1) \qquad \text{ for $i \in I_\af$.}
\end{align}
The $T_i$ satisfy $T_i^2=-T_i$ and the same braid relations as the $s_i$ do. Therefore for a reduced expression $w = s_{i_1}s_{i_2}\dotsm s_{i_\ell} \in W$ there are well defined elements $T_w = T_{i_1}T_{i_2}\dotsm T_{i_\ell}$.
Let $\K$ be the subring generated by $T_i$ for $i\in I_\af$ and $R(T)$.  We call it the {\it small-torus affine $K$-nilHecke ring}.


Let $\bL \subset \K$ be the centralizer of $R(T)$ in $\K$; this is called the $K$-Peterson subalgebra. The following theorem clarifies the geometric meaning of \cite[Theorems 5.3 and 5.4]{LSS}.  Recall that the ideal sheaf basis  $\{\xi_w \mid w \in W_\af^-\} \in K_0^T(\Gr)$ are the  unique elements characterized by $\ip{\cO^v}{\xi_w}_{\Gr} = \delta_{vw}$.

\begin{thm} \label{T:KPeterson}
There is an isomorphism of $R(T)$-Hopf algebras
$k:K_0^T(\Gr) \cong \bL$ such that for every $w\in W_\af^-$
\begin{enumerate}
\item[(a)]  the element $k_w := k(\xi_w)$
is the unique element in $\bL$ of the form
\begin{align}\label{eq:kw}
  k_w = T_w + \sum_{x\in W_\af\setminus W_\af^-} k_w^x T_x
\end{align}
where $k_w^x\in R(T)$, and
\item[(b)]  the element $l_w := k(\cO_w)$ is given by
\begin{equation}\label{eq:lw}
l_w = \sum_{v \leq w} k_w.
\end{equation}
\end{enumerate}
\end{thm}
\begin{proof}
In \cite{LSS}, a $K$-homology Hopf algebra $K_0^T(\Gr)$ was constructed as a Hopf dual to $K^0_T(\tGr)$.  In \cite[Theorem 5.3]{LSS}, an isomorphism $K_0^T(\Gr) \simeq \bL$ is constructed, and the $R(T)$-bilinear pairing $\ip{ \cdot }{ \cdot }_\bL: K^0_T(\tGr) \times \bL$ is given by \begin{equation}\label{eq:pairing2}
\ip{\cO^w}{a}_\bL = a_w \/, \end{equation} where $w \in W_\af^-$ and $a = \sum_{v \in W_\af} a_v T_v \in \bL \subset \K$ with $a_v \in R(T)$; see \cite[\S 2.4]{LSS}, especially equation (2.10). The uniqueness of the elements $k_w$ given by \eqref{eq:kw} is \cite[Theorem 5.4]{LSS}.

We now identify the $\bL$ with $K_0^T(\Gr)$ via \eqref{eq:pairing1} and \eqref{eq:pairing2}.  It follows from \cite[Theorem 5.4]{LSS} that under the resulting isomorphism $k:K_0^T(\Gr) \cong \bL$, we have $k(\xi_w) = k_w$.  Statement (b) follows immediately from Lemma \ref{lemma:xiO}.
\end{proof}

\subsection{Closed formula for structure constants}
For $x,y,z \in W_\af$, define the structure constants $c_{x,y}^z$ by
$$
\cO_x\cdot \cO_y = \sum_{z \in W_\af} c_{x,y}^z \cO_z
$$
with the product structure given by the isomorphism of Theorem \ref{T:KPeterson}.  We now give a closed formula for $c_{x,y}^z$ in terms of equivariant localizations.

Define the elements $y_i = 1 + T_i$ for $i \in I_\af$.  Then $y_i^2 = y_i$ and the $y_i$ satisfy the braid relations so that for $w \in W_\af$ we can define $y_w \in \K$.  The $\{y_w \mid w \in W_\af\}$ form a $R(T)$-basis of $\K$. For any   $q\in Q(T)$, we have {\upshape $qy_{s_i}=y_{s_i}(s_iq)+{q-s_iq\over 1-e^{-\alpha_i}} y_{\scriptsize\mbox{id}} .$}
 Define $b_{w,u} \in Q(T)$ and $e_{w, u} \in Q(T)$ respectively by\footnote{In the notation of \cite{KK:K}, they are denoted as $b_{u^{-1}, v^{-1}}$ and $e^{v^{-1}, u^{-1}}$ respectively.}
\begin{equation}\label{E:yw}
y_w = \sum_{u \in W_\af} b_{w,u} u,\qquad w=\sum_{u\in W_\af} e_{w, u}y_u \/. \end{equation}

The matrix $\big(b_{w, u}\big)$ is invertible, and its inverse is given by $\big(e_{w,u}\big)$.


\begin{prop}\label{claimcoeBBEE}
  Let {\upshape $u, v\in  W_{\scriptsize\mbox{af} }$}. Let
   $u=s_{\beta_1}\cdots s_{\beta_m}$ be a reduced expression of $u$. We have
     \begin{equation}\label{eqnbb}
        b_{u, v} =\sum \prod_{k=1}^m
             s_{\beta_1}^{\varepsilon_1}\cdots s_{\beta_{k-1}}^{\varepsilon_{k-1}}\big({(-e^{-\beta_k})^{\varepsilon_k}\over 1-e^{-\beta_k}}\big)\big|_{\alpha_0=-\theta},
     \end{equation}
           the summation over all   $(\varepsilon_1, \cdots,  \varepsilon_m)\!\in\!\! \{0, 1\}^m$ satisfying   $s_{\beta_1}^{\varepsilon_1}\cdots s_{\beta_m}^{\varepsilon_m}=v.$

      Denote $\gamma_j=s_{\beta_1}\cdots s_{\beta_{j-1}}(\beta_{j})$ for each $j$.   Then we have
         \begin{equation}\label{eqnee}
         e_{u, v} =\sum \prod_{k=1}^m\big((1-\varepsilon_k)e^{\gamma_k}+\varepsilon_k(1-e^{\gamma_k})\big)\big|_{\alpha_0=-\theta},
         \end{equation}
       the summation  over all   $(\varepsilon_1, \cdots,  \varepsilon_m)\!\in\!\! \{0, 1\}^m$ satisfying   $y_{s_{\beta_1}}^{\varepsilon_1}\cdots y_{s_{\beta_m}}^{\varepsilon_m}=y_v.$
    \end{prop}

In the present work, we work in $Q(T)$, where $T \subset G$ is the finite torus.  Our proof below holds in $Q(T_\af)$ where $T_\af$ denotes the affine torus.

\begin{proof}
The formula for $b_{u, v}$ follows immediately from the definitions.

The formula for $e_{u, v}$ holds by showing that both sides satisfy the same recursive formulas.  Precisely, let $\tilde{e}_{u, v}$ denote the RHS of \eqref{eqnee}. We shall show that $e_{u, v}$ and $\tilde{e}_{u, v}$ satisfy the same recursions.
  We have
   \begin{align*}
      {s_iu}&= (e^{\alpha_i} y_{\rm id}+(1-e^{\alpha_i})y_{s_i})\sum_v e_{u, v}y_v\\
                    &= \sum_ve^{\alpha_i}e_{u, v} y_{v}+\sum_v(1-e^{\alpha_i})y_{s_i} e_{u, v}y_v\\
                    &= \sum_ve^{\alpha_i}e_{u, v} y_{v}+\sum_v(1-e^{\alpha_i})(s_i(e_{u,v})y_{s_i}-{ s_i(e_{u, v})- e_{u, v}\over 1-e^{-\alpha_i}} y_{\rm id})y_v\\
                    &=\sum_ve^{\alpha_i}e_{u, v} y_{v}+\sum_v(1-e^{\alpha_i})s_i(e_{u,v})y_{s_i}y_v+\sum_v e^{\alpha_i}  (s_i(e_{u, v})-e_{u, v}) y_v\\
                    &=\sum_{v: s_iv<v}\big(e^{\alpha_i}e_{u, v}+(1-e^{\alpha_i})s_i(e_{u,v})+(1-e^{\alpha_i})s_i(e_{u,s_iv})+ e^{\alpha_i}(s_i(e_{u, v})-e_{u, v})\big)y_{v} \\
                    &\qquad+\sum_{v: s_iv>v}\big(e^{\alpha_i}e_{u, v}  + e^{\alpha_i}(s_i(e_{u, v})-e_{u, v})\big)y_{v}\\
                    &=\sum_{v: s_iv<v}\big(  s_i(e_{u,v})+(1-e^{\alpha_i})s_i(e_{u,s_iv}) \big)y_{v}
                             +\sum_{v: s_iv>v}  e^{\alpha_i}(s_i(e_{u, v}) y_{v})\\
   \end{align*}
 That is, for $s_iu<u$, we have
  \begin{equation}\label{eerecur}
     e_{s_iu, v}=\begin{cases}
        s_i(e_{u,v})+(1-e^{\alpha_i})s_i(e_{u,s_iv}),&\mbox{if } s_iv<v,\\
         e^{\alpha_i}s_i(e_{u, v}),&\mbox{if } s_iv>v.
   \end{cases}
  \end{equation}
 It follows directly from \eqref{eqnee} that   $\tilde{e}_{s_iu, v}$ satisfies the same recursive rule. Moreover, $e_{{\rm id}, v}=\delta_{{\rm id}, v}=\tilde{e}_{{\rm id}, v}$ for any $v$. Therefore the statement follows.
\end{proof}

Consider the left $Q(T)$-module homomorphism $\kk:Q(T) \otimes_{R(T)} \K \to Q(T) \otimes_{R(T)} \bL$ defined by
\begin{align*}
  \kk(t_\lambda w) = t_\lambda\qquad\text{for $w\in W$ and $\lambda\in Q^\vee$.}
\end{align*}

\begin{prop} \label{P:kappa}
The map $\kk$ restricts to a $R(T)$-module map $\kk:\K \to \bL$, and
\begin{align*}
\kk(T_u)&=0&&\text{if $u\in W_\af\backslash W_\af^-$.} \\
\kk(T_u)&=k_u&&\text{if $u\in W_\af^-$.} \\
\kk(y_u)&=l_u &&\text{if $u\in W_\af^-$.}
\end{align*}
\end{prop}
\begin{proof}
The first claim follows from the three formulas.  From the definition, $\kk(T_i) = 0$ for $i \neq 0$.  It follows easily that $\kk(T_u) = 0$ if $u \notin W_\af^-$.  By \cite[(5.1)]{LSS}, the element $k_u \in \bL$ can be characterized as follows.  Let $T_u = \sum_{x \in W_\af} a_x\, x$ for $a_x \in Q(T)$ and $k_u = \sum_{\lambda \in Q^\vee} a'_{t_\lambda} \, t_\lambda$ for $a'_{t_\lambda} \in Q(T)$.  Then for any function $f: W_\af \to R(T)$ satisfying $f(x) = f(xv)$ for $v \in W$, we have
$$
\sum_{x \in W_\af} a_x f(x) = \sum_{\lambda \in Q^\vee} a'_{t_\lambda} f(t_\lambda).
$$
It follows that $\kk(T_u) = k_u$.  The last claim follows from the first two formulas, the equality $y_w = \sum_{v \in W_\af,\, v \leq w} T_v$ and \eqref{eq:lw}.
\end{proof}

Denote
  $$b_{x, [y]}:= \sum_{z\in yW}b_{x, z},\qquad e_{x, [y]}:=\sum_{z\in yW}e_{x, z}.$$
\begin{thm}\label{T: coeffc} For any $x, y, z\in W_\af^-$, the coefficient $c_{x, y}^z$ is given by
\begin{equation}\label{eq:bbe}
   c_{x,y}^z=\sum_{t_1, t_2\in Q^\vee}b_{x, [t_1]}b_{y, [t_2]}e_{t_1t_2, [z]}.
\end{equation}
\end{thm}

\begin{proof} By Theorem \ref{T:KPeterson}, we have
$ l_xl_y= \sum_{z\in W_\af^-}c_{x,y}^zl_z$.
By Proposition \ref{P:kappa}, we have
   \begin{align*}
       l_xl_y&=\sum_{u, v\in W_\af} \kappa(b_{x, u} u)\kappa(b_{y, v}v)\\
       &=\sum_{t_1, t_2\in Q^\vee}\sum_{u, v\in W} b_{x, t_1u}b_{y, t_2v}\kappa(t_1u)\kappa(t_2v) \\
       &=\sum_{t_1, t_2\in Q^\vee}\sum_{u, v\in W} b_{x, t_1u}b_{y, t_2v} t_1 t_2  \\
        &=\sum_{z\in W_\af^-}\sum_{t_1, t_2\in Q^\vee} b_{x, [t_1]}b_{y, [t_2]} e_{t_1 t_2, [z]}l_z. \qedhere
   \end{align*}
\end{proof}

{  \subsection{Geometric remarks} We will provide a brief geometric interpretation of the previous approach.~There is an $R(T)$-module identification $\mathbb{K} = K_0^T(\Fl)$ and an $R(T)$-Hopf algebra identification $\mathbb{L} = K_0^T(\Gr)$.~The classes $y_w \in \K$ play two roles: on one side $y_w = \cO_w$ are the
structure (finite dimension) Schubert structure sheaves on the affine flag manifold $\Fl$; on the other side they {\em act} on $K^0_T(\Fl)$ as the K-theoretic BGG operators $\partial_w$ - see \cite[Lemma 2.2]{LSS}. Similarly, the elements $T_w \in \K$ correspond to the ideal sheaves $\xi_w$ on $K_0(\Fl)$, or to the BGG-type operators $\partial_w - id$. The map $\kappa: \K \to \bL$ is the K-theoretic projection map $\pi_*: K_0^T(\Fl) \to K_0^T(\Gr)$, and the classes $k_w$ and $l_w$ (for $w \in W_{\rm af}^-$) correspond respectively to the ideal sheaves and Schubert structure sheaves in the affine Grassmannian. In particular, Proposition \ref{P:kappa} states that $$\pi_*(\xi_w^{\Fl}) = \begin{cases} \xi_{w} & \textrm{if } w \in W_{\af}^- \\ 0 & \textrm{ otherwise } \end{cases};\qquad \pi_*(\cO_u^{\Fl})=\cO_u\,\, \mbox{ for }\,\,u\in W_{\rm af}^-.$$
It is not difficult to prove these identities directly, using Lemma \ref{lemma:xiO} and identities (\ref{E:1}), (\ref{E:2}).

For each of $K_0^T(\Fl)$ and $K_0^T(\Gr)$, there is a third basis $ \{ \iota_w \}$, indexed respectively by $W_{\rm af}$ and by $W_{\rm af}/ W$, called the {\em localization basis}. If $w \in W_{\rm af}$ then $\iota_w \in K_0^T(\Fl)$ is the map $\iota_w: K^0_T(\Fl) \to R(T)$ defined by sending the $K$-cohomology class $\cO^u$ to its localization to the fixed point $w$. Then equation (\ref{E:yw}) above corresponds to expanding the structure sheaf basis into localization basis and viceversa. A key observation from \cite{LL} and \cite{LSS}, which is used in the proof of Theorem \ref{T: coeffc}, is that the Pontryagin multiplication on $K_0^T(\Gr)$ is easy to write in the localization basis: if $\lambda, \mu \in Q^\vee$ and $\iota_{t_\lambda}, \iota_{t_\mu} \in K_0^T(\Gr)$ are the corrsponding localization elements, then $\iota_{t_\lambda} \cdot \iota_{t_\mu} = \iota_{t_{\lambda + \mu}}$; see \cite[Lemma 5.1]{LSS}.}

 \section{Data and Evidence} {  As we observed above, the cohomological versions of Conjectures \ref{conj:constants} and \ref{conj:isomorphism} were proved in \cite{LS}. In the $K$-theoretic version, we can verify Conjecture \ref{conj:constants} when the degree $d$ in $N_{u, v}^{w, d}$ is $d=0$ or $d= \alpha_i^\vee$ is a simple coroot. Our arguments are similar to those in \cite{LL}, but are quite involved, even in these situations. It would be desirable to find more conceptual explanations.}
Next we provide two computational examples.
\subsection{Conjecture is true for $G=SL_2$}
The complete flag manifold $SL_2/B$ is the complex projective line $\mathbb{P}^1$. The Weyl group $W=\mathbb{Z}_2$   is generated by the   simple reflection $s_1=s_{\alpha}$ of the unique simple root $\alpha=\alpha_1$.  The equivariant quantum $K$-theory $QK_T(\mathbb{P}^1)$ has an $R(T)[q]$-basis $\{\cO^{\rm id}, \cO^{s_1}\}$. As shown in \cite{buch.m:qk}, the only nontrivial quantum product is given by\footnote{We   use the opposite identification $e^{\varepsilon_i}=-[\mathbb{C}_{\varepsilon_i}]\in R(T)$ compared with \cite[Section 5.5]{buch.m:qk}.}
 \begin{equation}\label{qKP1}
    \cO^{s_1}\star \cO^{s_1}=(1-e^{-\alpha}) \cO^{s_1}+e^{-\alpha} q.
 \end{equation}

On the affine side, we notice that $s_0=s_1t_{-\alpha^\vee}$ and that $$W_{\rm af}^-=\{\mbox{id}\}\cup \{wt_{n \alpha^\vee}~|~ n\in \mathbb{Z}_{<0}, w={\rm id} \mbox{ or } s_1\}.$$
Let $g_m$ be the unique element of $W_\af^-$ of length $m$ for $m\ge0$.
Let $h_m$ be the unique element of $W_\af\setminus W_\af^-$ of length $m$ for $m \ge 1$.  For example, $g_3 = s_0s_1s_0$ and $h_4 = s_0 s_1 s_0 s_1$.  Notice that $T_i f=s_i(f)T_i+T_i(f)$ and $T_i^2=-T_i$ for any $f\in R(T)$ and $i\in \{0,1\}$.

\begin{lemma} \label{lemma:kx} We have $k_{\rm id}=1$.
For $r\ge1$, we have
$$
	k_{g_{2r-1}} = T_{g_{2r-1}} + T_{h_{2r-1}} + (1-e^{-\alpha}) T_{h_{2r}} \qquad \text{and} \qquad
    k_{g_{2r}} = T_{g_{2r}} + e^{-\alpha} T_{h_{2r}}.
$$
\end{lemma}
\begin{proof}
  Denote by $\tilde k_{g_{m}}$ the expected formula. By Theorem \ref{T:KPeterson}(a), it suffices to show  $\tilde k_{g_{m}}\in \bL$, or equivalently, $\tilde k_{g_{m}} e^{-\alpha}=e^{-\alpha}\tilde k_{g_{m}}$. Clearly, this holds when $m=0$. It also holds for $m\in \{1, 2\}$ by direct calculations. In particular  we have
   \begin{align*}
       (T_0+T_1+(1-e^{-\alpha})T_{01})e^{\pm\alpha}&=e^{\pm\alpha}(T_0+T_1+(1-e^{-\alpha})T_{01});\\
       (T_{10}+e^{-\alpha}T_{01})e^{\pm\alpha}&=e^{\pm\alpha}(T_{10}+e^{-\alpha}T_{01}).
    \end{align*}  Assume that it holds for $m\leq 2r$ where $r\geq 1$. Then we have
    \begin{align*}
      \tilde k_{g_{2r+1}} e^{-\alpha}&=\big(T_{g_{2r+1}}+T_1T_{h_{2r}}+(1-e^{-\alpha})T_{01}T_{h_{2r}}\big)e^{-\alpha}\\
                                       &= T_{g_{2r+1}}e^{-\alpha}+(T_1+(1-e^{-\alpha})T_{01})e^{\alpha}e^{-\alpha}T_{h_{2r}}e^{-\alpha}\\
                                       &=T_{g_{2r+1}}e^{-\alpha}+\big(e^{\alpha}(T_0+T_1+(1-e^{-\alpha})T_{01})-T_0e^{\alpha}\big)
                                       \big(e^{-\alpha}(T_{g_{2r}}+e^{-\alpha}T_{h_{2r}})-T_{g_{2r}}e^{-\alpha}\big)\\
                                       &=  T_{g_{2r+1}}e^{-\alpha}+\big(T_0+T_1+(1-e^{-\alpha})T_{01}\big)T_{g_{2r}}+e^{-\alpha}\big(T_0+T_1+(1-e^{-\alpha})T_{01}\big)T_{h_{2r}}\\
                                           &\qquad-T_0(T_{g_{2r}}+e^{-\alpha}T_{h_{2r}})+T_0e^{\alpha}T_{g_{2r}}e^{-\alpha}
                                           -e^{\alpha}\big(T_0+T_1+(1-e^{-\alpha})T_{01}\big)T_{g_{2r}}e^{-\alpha}\\
                                           &=e^{-\alpha}\tilde k_{g_{2r+1}}+T_1T_{g_{2r}}+e^{-\alpha}T_0T_{h_{2r}}-T_0T_{h_{2r}}e^{-\alpha}-
                                           e^{\alpha}T_1T_{g_{2r}}e^{-\alpha}\\
                                           &=e^{-\alpha}\tilde k_{g_{2r+1}}-(T_{g_{2r}}+e^{-\alpha}T_{h_{2r}})e^{\alpha}e^{-\alpha}+T_{h_{2r}}e^{-\alpha}+
                                           e^{\alpha}T_{g_{2r}}e^{-\alpha}\\
                                           &=e^{-\alpha}\tilde k_{g_{2r+1}}
    \end{align*}
    Similarly, we can show  $\tilde k_{g_{2r+2}} e^{-\alpha}=   e^{-\alpha} \tilde k_{g_{2r+2}}$. Thus the statement follows.
\end{proof}

The following result follows from Lemma \ref{lemma:kx} and Theorem \ref{T:KPeterson}(b).

\begin{lemma} \label{lemma:lx} We have $l_{\rm id}=1$.
For $r\ge1$, we have
$$
	\ell_{g_{2r-1}} = (1-e^{-\alpha}) T_{h_{2r}} + \sum_{\substack{v\in W_\af \\ \ell(v) \le 2r-1}} T_v \qquad \text{and} \qquad
	\ell_{g_{2r}} = \sum_{\substack{v\in W_\af \\ \ell(v) \le 2r}} T_v.
$$
\end{lemma}
  \begin{prop} \label{P:ltrans}
For $x\in W_\af^-$ and $n\in \mathbb{Z}_{<0}$, we have in $K_0^T(\Gr_{\SL_2})$
\begin{align*}
    \cO_x \cdot \cO_{t_{n\alpha^\vee}}=\cO_{x t_{n\alpha^\vee}}.
\end{align*}
\end{prop}
\begin{proof} It suffices to prove the statement for $n=-1$. Notice that $t_{-\alpha^\vee}=s_1s_0=g_2$ and  $x=g_m$ for some $m\in\mathbb{Z}_{\geq 0}$. By Theorem \ref{T:KPeterson}, we just need to show $l_{g_m}l_{g_2}=l_{g_{m+2}}$. This follows from Lemma \ref{lemma:lx} and mathematical induction on $m$.
\end{proof}
Thanks to the above formula, it remains to compute    $\cO_{s_1t_{-\alpha^\vee}} \cdot\cO_{s_1t_{-\alpha^\vee}}$.
For $x=s_1t_{-\alpha^\vee}=s_0=g_1$, by direct calculations we have $l_{g_1}^2=e^{-\alpha}l_{g_2}+(1-e^{-\alpha})l_{g_3}$. Therefore
  \begin{equation}\label{affP1}
       \cO_{s_1t_{-\alpha^\vee}} \cdot\cO_{s_1t_{-\alpha^\vee}}=(1-e^{-\alpha})\cO_{s_1t_{-2\alpha^\vee}}+e^{-\alpha} \cO_{t_{-\alpha^\vee}}.
   \end{equation}

\begin{remark}
  We can also calculate the above product by using Theorem \ref{T: coeffc}. For instance, for $z=s_{1}s_0=t_{-\alpha^\vee}$, all the  terms in the formula \eqref{eq:bbe} for   $c_{x,x}^z$ vanish  unless $t_1=t_2=t_{\alpha^\vee}$. Therefore
    $$c_{x, x}^{z}= b_{s_0, [t_{\alpha^\vee}]}^2 e_{t_{2\alpha^\vee}, [s_1s_0]}= \big({-e^{\alpha}\over 1-e^{ \alpha}}\big)^2e^{-\alpha}(1-e^{-\alpha})^2=e^{-\alpha}.$$
\end{remark}

 Formulas \eqref{qKP1} and \eqref{affP1}, together with Proposition \ref{P:ltrans}, implies that Conjectures \ref{conj:constants} and \ref{conj:isomorphism} hold when $G=\SL_2$.
\subsection{Multiplication for $\Gr_{SL_3}$}The complete flag manifold $\SL_3 /B=\Fl_3=\{V_1\subset V_2 \subset \mathbb{C}^3~|~ \dim V_1=1, \dim V_2=2\}$
parameterizes complete flags in $\mathbb{C}^3$.  The Weyl group $W$   is the permutation group $S_3$ of three objects generated by  simple reflections $s_1, s_2$. We have the highest root $\theta=\alpha_1+\alpha_2$ and coroot $\theta^\vee=\alpha_1^\vee+\alpha_2^\vee$. By calculations using Theorem \ref{T: coeffc}, we obtain
   $\cO_{wt_{-\theta^\vee}}\cO_{t_{-\theta^\vee}}=\cO_{wt_{-2\theta^\vee}}$ in $K_0^T(\Gr_{\SL_3})$ for any $w\in W$, in addition to the following multiplication table.

 \resizebox{.99\hsize}{!}{\parbox{7cm}{\begin{align*}
   \cO_{s_1t_{-\theta^\vee}}  \cO_{s_1t_{-\theta^\vee}}&= (1-e^{-\alpha_1})\cO_{s_1t_{-2\theta^\vee}}+e^{-\alpha_1}\cO_{s_2s_1t_{-2\theta^\vee}}+e^{-\alpha_1}\cO_{t_{-2\theta^\vee+\alpha_1^\vee}} -e^{-\alpha_1}\cO_{s_2t_{-2\theta^\vee+\alpha_1^\vee}}\\
   \cO_{s_1t_{-\theta^\vee}}  \cO_{s_2t_{-\theta^\vee}}&=\cO_{s_1s_2 t_{-2\theta^\vee}}+\cO_{s_2s_1 t_{-2\theta^\vee}}-\cO_{s_1s_2s_1 t_{-2\theta^\vee}}\\
   \cO_{s_1t_{-\theta^\vee}}  \cO_{s_1s_2t_{-\theta^\vee}}&=(1-e^{-\alpha_1})\cO_{s_1s_2 t_{-2\theta^\vee}}+e^{-\alpha_1}\cO_{s_1s_2s_1 t_{-2\theta^\vee}}\\
     \cO_{s_1t_{-\theta^\vee}}  \cO_{s_2s_1t_{-\theta^\vee}}&=(1-e^{-\alpha_1-\alpha_2})\cO_{s_2s_1 t_{-2\theta^\vee}}+e^{-\alpha_1-\alpha_2} \cO_{s_2 t_{-2\theta^\vee+\alpha_1^\vee}}\\
     \cO_{s_1t_{-\theta^\vee}}  \cO_{s_1s_2s_1t_{-\theta^\vee}}&=(1-e^{-\alpha_1-\alpha_2})\cO_{s_1s_2s_1 t_{-2\theta^\vee}}+e^{-\alpha_1-\alpha_2} \cO_{s_1s_2 t_{-2\theta^\vee+\alpha_1^\vee}}\\
       &\quad+ e^{-\alpha_1-\alpha_2}\cO_{t_{-2\theta^\vee+\alpha_1^\vee+\alpha_2^\vee}}-e^{-\alpha_1-\alpha_2}\cO_{s_1t_{-2\theta^\vee+\alpha_1^\vee+\alpha_2^\vee}}\\
      \cO_{s_1s_2t_{-\theta^\vee}}  \cO_{s_1s_2t_{-\theta^\vee}}&=(1-e^{-\alpha_1})(1-e^{-\alpha_1-\alpha_2})\cO_{s_1s_2t_{-2\theta^\vee}}+e^{-\alpha_1} \cO_{s_2s_1t_{-2\theta^\vee+\alpha_2^\vee}}\\
         &\quad +(1-e^{-\alpha_1})e^{-\alpha_1-\alpha_2}\cO_{s_1t_{-2\theta^\vee+\alpha_2^\vee}}\\
   \cO_{s_1s_2t_{-\theta^\vee}}  \cO_{s_2s_1t_{-\theta^\vee}}&=(1-e^{-\alpha_1-\alpha_2})\cO_{s_1s_2s_1t_{-2\theta^\vee}}+e^{-\alpha_1-\alpha_2} \cO_{t_{-2\theta^\vee+\alpha_1^\vee+\alpha_2^\vee}}\\
    \cO_{s_1s_2t_{-\theta^\vee}}  \cO_{s_1s_2s_1t_{-\theta^\vee}}&=(1-e^{-\alpha_1})(1-e^{-\alpha_1-\alpha_2})\cO_{s_1s_2s_1t_{-2\theta^\vee}}+(1-e^{-\alpha_1-\alpha_2})e^{-\alpha_1} \cO_{s_2s_1t_{-2\theta^\vee+\alpha_2^\vee}}\\
                 &\quad+(1-e^{-\alpha_1}) e^{-\alpha_1-\alpha_2} \cO_{ t_{-2\theta^\vee+\alpha_1^\vee+\alpha_2^\vee}}+  e^{-2\alpha_1-\alpha_2}  \cO_{s_2 t_{-2\theta^\vee+\alpha_1^\vee+\alpha_2^\vee}}\\
 \cO_{s_1s_2s_1t_{-\theta^\vee}}   \cO_{s_1s_2s_1t_{-\theta^\vee}}&=(1-e^{-\alpha_1})(1-e^{-\alpha_2})(1-e^{-\alpha_1-\alpha_2})\cO_{s_1s_2s_1t_{-2\theta^\vee}}+e^{-\alpha_1-\alpha_2}\cO_{s_1s_2t_{-2\theta^\vee+\alpha_1^\vee+\alpha_2^\vee}}\\
                       &\quad +(1-e^{-\alpha_1})(1-e^{-\alpha_1-\alpha_2})e^{-\alpha_2} \cO_{s_1s_2t_{-2\theta^\vee+\alpha_1^\vee}}+e^{-\alpha_1-\alpha_2}\cO_{s_2s_1t_{-2\theta^\vee+\alpha_1^\vee+\alpha_2^\vee}}\\
                         &\quad      + (1-e^{-\alpha_2})(1-e^{-\alpha_1-\alpha_2}) \cO_{s_2s_1t_{-2\theta^\vee+\alpha_2^\vee}}-e^{-\alpha_1-\alpha_2}\cO_{s_1s_2s_1t_{-2\theta^\vee+\alpha_1^\vee+\alpha_2^\vee}}\\
                         & \quad+e^{-\alpha_1-2\alpha_2}(1-e^{-\alpha_1^\vee})\cO_{ s_1t_{-2\theta^\vee+\alpha_1^\vee+\alpha_2^\vee}}
                         +e^{-2\alpha_1- \alpha_2}(1-e^{-\alpha_2^\vee})\cO_{ s_2t_{-2\theta^\vee+\alpha_1^\vee+\alpha_2^\vee}}\\
                         &\quad+ e^{-\alpha_1- \alpha_2}(1-e^{-\alpha_1^\vee})(1-e^{-\alpha_2^\vee})\cO_{ t_{-2\theta^\vee+\alpha_1^\vee+\alpha_2^\vee}}\\
  \end{align*}}}
 The remaining products are read off immediately from the above table by the symmetry of the Dynkin diagram of Lie type $A_2$.

 Comparing the above table with the appendix in \cite{li.mihalcea:lines}, we conclude that Conjecture \ref{conj:constants} holds whenever the degree $d$ in $N_{u, v}^{w, d}$ is given by $(0,0), (1,0)$ or $(0, 1)$.

\end{document}